\documentclass[reqno,b5paper]{amsart}
\usepackage{amsmath}
\usepackage{amssymb}
\usepackage{amsthm}
\usepackage{enumerate}
\usepackage[mathscr]{eucal}

\newtheorem{theorem}{Theorem}[section]
\newtheorem{claim}{Claim}[section]

\begin{document}

\title[An embedding in the Fell topology]{An embedding in the Fell topology}

\newcommand{\acr}{\newline\indent}

\address{\llap{*}Academy of Sciences, Institute of Mathematics \acr \v
Stef\'anikova 49,
81473 Bratislava, Slovakia
\acr Slovakia}

\email{hola@mat.savba.sk}

\thanks{}

\keywords{embedding, hyperspace, Fell topology, Wijsman topology,  $\omega_1$.
The author  would like to thank to grant APVV-0269-11 and  Vega 2/0018/13.}

\begin{abstract}

Using a result from [HZ] it is shown that if a $T_2$ topological space $X$ contains a closed uncountable discrete subspace, then  the space $\omega_1 \times (\omega_1 + 1)$ embeds  as a closed subspace of $(CL(X),\tau_F)$, the hyperspace of nonempty closed subsets of $X$ equipped  with the Fell topology. We will use the above result to give a partial answer to the Question 3.8 in [CJ].

\end{abstract}

\maketitle

\section{Introduction}
\bigskip
\bigskip
\bigskip

Throughout this paper, let $CL(X)$ denote the family of all nonempty closed subsets of a given $T_2$ topological space. For $M \in CL(X)$, put

\bigskip
\centerline{$M^- = \{A \in CL(X): A \cap M \ne \emptyset\}$, $M^+ = \{A \in CL(X): A \subseteq M\}$}

\bigskip
and denote $M^c = X \setminus M$. The Vietoris topology [Mi] $\tau_V$ on $CL(X)$ has as subbase elements of the form $U^-$, $V^+$, where $U, V$ are open in $X$.

\bigskip
The Fell topology [Fe] $\tau_F$ on $CL(X)$ has as a subbase the collection

\bigskip
\centerline{$\{U^-: U$ open in $X\} \cup \{(K^c)^+: K$ compact in $X\}$.}

\bigskip
It si known [Be] that $(CL(X),\tau_F)$ is Hausdorff (regular, Tychonoff, respectively) iff $X$ is locally compact.

\bigskip

For a metric space $(X,d)$, let $d(x,A) =$ inf$\{d(x,a): a \in A\}$ denote the distance between a point $x \in X$ and a nonempty subset $A$ of $(X,d)$.

A net $\{A_\alpha: \alpha \in \lambda\}$ in $CL(X)$ is said to be Wijsman convergent to some $A$ in $CL(X)$ if $d(x,A_\alpha)\rightarrow d(x,A)$ for every $x \in X$. The Wijsman topology on $CL(X)$ induced by $d$, denoted by $\tau_{w(d)}$, is the weakest topology such that for every $x \in X$, the distance functional

\bigskip
\centerline{$d(x,.): CL(X)  \to R^+$}
\bigskip
is continuous. It can be seen easily that the Wijsman topology on $CL(X)$ induced by $d$ has the family

\bigskip
\centerline{$\{U^-: U$ open in $X\} \cup \{\{A \in CL(X): d(x,A) > \epsilon\}: x \in X, \epsilon > 0\}$}

\bigskip
as a subbase [Be].

The above type of convergence was introduced by Wijsman in [Wi] for sequences of closed convex sets in Euclidean space $R^n$, when he considered optimum properties of the sequential probability ratio test.

\bigskip
\bigskip

\section{Main result}
\bigskip
\bigskip

\begin{claim} ([HZ]) Let $X$ be a $T_2$ topological space which contains an uncountable closed discrete set. Then $\omega_1$ embeds into $(CL(X),\tau_F)$ as a closed set.

\end{claim}

\begin{proof}

Let $D = \{x_\alpha: \alpha < \omega_1\}$ be an uncountable closed discrete set in $X$. Put $E_0 = X$ and

\bigskip
\centerline{$E_\alpha = \{x_\eta: \alpha \le \eta < \omega_1\}$ for $\alpha \ne 0, \alpha < \omega_1$.}

\bigskip
Then for every $\alpha < \beta < \omega_1$, $E_\beta \subset E_\alpha$, $E_\alpha = \bigcap_{\beta < \alpha} E_\beta$ for every limit ordinal $\alpha$ and $\bigcap_{\alpha < \omega_1} E_\alpha = \emptyset$.

\bigskip
We will define the mapping $\varphi: \omega_1 \to (CL(X),\tau_F)$ as follows:  $\varphi(\alpha) = E_\alpha$ for every $\alpha < \omega_1$.

We will show that the mapping $\varphi$ is a homeomorphism between $\omega_1$ and $\{E_\alpha: \alpha < \omega_1\}$. Of course $\varphi$ is injective.

First we show that $\varphi$ is continuous. Let $\alpha < \omega_1$ be a limit ordinal. Then $E_\alpha = \bigcap_{\beta < \alpha} E_\beta$. Since $E_\alpha \subset E_\beta$ for every $\beta < \alpha$, it is sufficient to consider $E_\alpha \in (K^c)^+$ for a compact subset $K \subset X$. There must exist $\eta < \alpha$ such that $E_\eta \cap K = \emptyset$, otherwise $E_\alpha \cap K \ne \emptyset$. Thus $\varphi((\eta,\alpha]) \subset (K^c)^+$.

Now we show that $\varphi$ is open. It is easy to verify that for $\alpha$ isolated, $\varphi(\alpha)$ is isolated in $\varphi(\omega_1)$. Let $\alpha$ be a limit ordinal and consider an open interval $(\beta,\alpha]$. Let $U(x_\alpha)$ be an open neighbourhood of $x_\alpha$ such that $U(x_\alpha) \cap D = \{x_\alpha\}$. Then $U(x_\alpha)^- \cap (\{x_\beta\}^c)^+ \cap \varphi(\omega_1)$ is open in $\varphi(\omega_1)$ and

\bigskip
\centerline{$\varphi((\beta,\alpha]) = U(x_\alpha)^- \cap (\{x_\beta\}^c)^+ \cap \varphi(\omega_1)$.}

\bigskip
We will show that $\varphi(\omega_1)( = \{E_\alpha: \alpha < \omega_1\})$ is closed in $(CL(X),\tau_F)$. Let $A \in CL(X) \setminus \varphi(\omega_1)$. Suppose first that $A \cap (X \setminus E_1) \ne \emptyset$. Let $x \in X \setminus A$. Put $\mathcal U = (X \setminus E_1)^- \cap (\{x\}^c)^+$. Then $A \in \mathcal U$ and $\mathcal U \cap \varphi(\omega_1) = \emptyset$.

Suppose now that $A \subset E_1$. Put $\alpha_0 =$ min$ \{\beta < \omega_1: E_\beta \nsupseteq A\}$. Recall $\bigcap_{\alpha < \omega_1} E_\alpha = \emptyset$ and $A \ne \emptyset$. Clearly, $\alpha_0$ cannot be limit, so $\alpha_0 = \delta + 1$. Thus $x_\delta \in A$ and there must exist $\eta \ge \alpha_0$ such that $x_\eta \notin A$. Let $U(x_\delta)$ be an open neighbourhood of $x_\delta$ such that $U(x_\delta) \cap D = \{x_\delta\}$.  Put $\mathcal U = U(x_\delta)^- \cap (\{x_\eta\}^c)^+$. Then $A \in \mathcal U$ and $\mathcal U \cap \varphi(\omega_1) = \emptyset$.

\end{proof}

\bigskip
\begin{claim} Let $X$ be a  $T_2$ topological space which contains an uncountable closed discrete set. Then $\omega_1 + 1$ embeds into $(CL(X),\tau_F)$ as a closed set.

\end{claim}
\begin{proof} Let $D = \{x_\alpha: \alpha < \omega_1\}$ be an uncountable closed discrete set in $X$ as above and also let $E_\alpha$, $\alpha < \omega_1$ be sets defined in the proof of Claim 1. For every $\alpha \le \omega_1$, define the sets $F_\alpha$ as follows:

\bigskip

\centerline{$F_\alpha = E_\alpha \cup \{x_0\}, \alpha < \omega_1$ and $F_{\omega_1} = \{x_0\}$.}

\bigskip
Define the mapping $\eta: \omega_1 + 1 \to (CL(X),\tau_F)$ as follows: $\eta(\alpha) = F_\alpha$, $\alpha \le \omega_1$.

Using the proof of the Claim 2.1 it is easy to verify that $\eta$ embeds $\omega_1 + 1$ into $(CL(X),\tau_F)$ as a closed subspace.

\end{proof}

\bigskip
\begin{theorem} Let $X$ be a $T_2$ topological space which contains an uncountable closed discrete set. Then $\omega_1 \times (\omega_1 + 1)$ embeds into $(CL(X),\tau_F)$ as a closed set.
\end{theorem}
\begin{proof} Let $D$ be an uncountable closed discrete set in $X$. We express $D$ as the disjoint union

\bigskip
\centerline{$D = D_0 \cup D_1$,}

\bigskip
such that $ \mid D_0 \mid = \aleph_1 = \mid D_1 \mid$. We enumerate $D_0 = \{x_\alpha: \alpha < \omega_1\}$ and $D_1 = \{y_\alpha: \alpha < \omega_1\}$ and put $E_0 = X$, $E_\alpha = \{x_\beta: \alpha \le \beta < \omega_1\}$, $\alpha \ne 0$, $\alpha < \omega_1$, $F_0 = X$, $F_\alpha = \{y_\beta: \alpha \le \beta < \omega_1\} \cup \{y_0\}$, $\alpha \ne 0, \alpha < \omega_1$ and $F_{\omega_1} = \{y_0\}$.

Define now the mappings $\varphi$ and $\eta$ as in the above Claims:

\bigskip
\centerline{$\varphi: \omega_1 \to (CL(X),\tau_F)$, $\varphi(\alpha) = E_\alpha$, $\alpha < \omega_1$,}
\bigskip
\centerline{$\eta:\omega_1 + 1 \to (CL(X),\tau_F)$, $\eta(\alpha) = F_\alpha$, $\alpha \le \omega_1$}

\bigskip
and define $\Pi: \omega_1 \times (\omega_1 + 1) \to (CL(X),\tau_F)$ as follows: $\Pi(0,0) = X$, $\Pi(0,\beta) = \eta(\beta)$, $\beta \ne 0, \beta \le \omega_1$, $\Pi(\alpha,0) = \varphi(\alpha)$, $\alpha \ne 0, \alpha < \omega_1$ and

\bigskip
\centerline{$\Pi(\alpha,\beta) = \varphi(\alpha) \cup \eta(\beta)$, $\alpha \ne 0 \ne \beta, \alpha < \omega_1, \beta \le \omega_1$.}

\bigskip
Put $\mathcal L =  \Pi(\omega_1 \times (\omega_1 + 1))$. We claim that $\Pi$ is an embedding from $\omega_1 \times (\omega_1 + 1)$ onto $(\mathcal L,\tau_F)$ and $\mathcal L$ is a closed subspace of $(CL(X),\tau_F)$.

\bigskip
Clearly, $\Pi$ is one-to-one. Using the continuity of $\varphi: \omega_1 \to (\varphi(\omega_1),\tau_F)$ and the continuity of $\eta: \omega_1 + 1 \to (\eta(\omega_1 + 1),\tau_F)$ it is easy to verify that also $\Pi: \omega_1 \times (\omega_1 + 1) \to (\mathcal L,\tau_F)$ is continuous. Also using the openess of $\varphi: \omega_1 \to (\varphi(\omega_1),\tau_F)$ and the openess of $\eta: \omega_1 + 1 \to (\eta(\omega_1 + 1),\tau_F)$ it is easy to see that $\Pi: \omega_1 \times (\omega_1 + 1) \to (\mathcal L,\tau_F)$ is open.

We will show that $\mathcal L$ is closed in $(CL(X),\tau_F)$. Let $A \in CL(X) \setminus \mathcal L$. Suppose first that
$A \cap (X \setminus (E_1 \cup D_1)) \ne \emptyset$. Let $x \in X \setminus A$. Put

\bigskip
\centerline{$\mathcal V = (X \setminus (E_1 \cup D_1))^- \cap (\{x\}^c)^+$.}

\bigskip
Then $A \in \mathcal V$ and $\mathcal V \cap \mathcal L = \emptyset$.

Suppose now that $A \subset E_1 \cup D_1$. First, suppose that $A \subset E_1$. We use the proof of Claim 2.1, where for $\mathcal U = U(x_\delta)^- \cap (\{x_\eta\}^c)^+$, $\delta < \eta$, we have $A \in \mathcal U$ and $\mathcal U \cap \varphi(\omega_1) = \emptyset$. Without loss of generality we can suppose that $U(x_\delta)$ is an open neighbourhood of $x_\delta$ such that $U(x_\delta) \cap (D_0 \cup D_1) = \{x_\delta\}$.  We claim that $\mathcal U \cap \mathcal L = \emptyset$.
Suppose there is some $(\alpha,\beta)$ with $\Pi((\alpha,\beta)) \in \mathcal U$, then $\varphi(\alpha) \in U(x_\delta)^- \cap (\{x_\eta\}^c)^+$, a contradiction. Similarly, if $A \subset D_1$.

Suppose now, that $A \cap E_1 \ne \emptyset$ and $A \cap D_1 \ne \emptyset$. Then either $A \cap E_1 \notin \varphi(\omega_1)$ or $A \cap D_1 \notin \eta(\omega_1 + 1)$. So we can again use the  above idea.

\end{proof}

\section{Embedding in the Wijsman topology}

In this part we will give a partial answer to the Question 3.8 in [CJ].

\bigskip

{\bf Ouestion 3.8 [CJ]} Let $(X,d)$ be a metric space. If $(CL(X),\tau_{w(d)})$ is non-normal, does $(CL(X),\tau_{w(d)})$ contain a closed copy of $\omega_1 \times (\omega_1 + 1)$?
\bigskip

It is known [Be] that if a metric space $(X,d)$ has nice closed balls, then the Fell topology $\tau_F$ and the Wijsman topology $\tau_{w(d)}$ on $CL(X)$ coincide.

A metric space $(X,d)$ is said to have nice closed balls [Be] provided whenever $B$ is a closed ball in $X$ that is a proper subset of $X$, then $B$ is compact.

It is also known [LL] that if $(X,d)$ is a metric space, then $(CL(X),\tau_{w(d)})$ is metrizable if and only if $(X,d)$ is separable.

\bigskip
\begin{theorem} Let $(X,d)$ be a metric space with nice closed balls. If $(CL(X),\tau_{w(d)})$ is non-normal, then $(CL(X),\tau_{w(d)})$ contains a closed copy of $\omega_1 \times (\omega_1 + 1)$.

\end{theorem}

\bigskip

Moreover we have the following theorem which gives a better  partial answer to the Question 3.8 in [CJ].

\begin{theorem} Let $(X,d)$ be a metric space such that every closed proper ball in $X$  is totally bounded. If $(X,d)$ is non-separable, then $(CL(X),\tau_{w(d)})$ contains a closed copy of $\omega_1 \times (\omega_1 + 1)$.

\end{theorem}
\begin{proof} Since $(X,d)$ is non-separable, there exist $\epsilon > 0$ and a set $D \subset X$ with $\mid D \mid = \aleph_1$ which is $\epsilon$-discrete, that is,  $d(x,y) \ge \epsilon$ for all distinct $x, y \in D$. We express $D$ as a disjoint union $D = D_0 \cup D_1$ as in the proof of Theorem 2.1 and we will proceed as in the proof of Theorem 2.1. We claim that $\Pi : \omega_1 \times (\omega_1 + 1) \to (CL(X),\tau_{w(d)})$ is embedding. Of course, it is sufficient to prove that $\Pi$ is continuous.

It is sufficient to verify that if $d(x,\Pi((\alpha_0,\beta_0)) > r$ for some $x \in X$ and $r > 0$, then
$d(x,\Pi((\alpha,\beta)) > r$ for all $(\alpha,\beta)$ from a neighbourhood of $(\alpha_0,\beta_0)$. However, it is clear, since the closed proper ball with center $x$ and the radius $r$ can contain only finitely many points of the set $D$.

\end{proof}

\bigskip

Theorem 3.2 gives another proof of the known result (Corollary 5.8, [HN]), which claims,  that if $(X,d)$ is a metric space in which every closed proper ball is totally bounded, then $(CL(X),\tau_{w(d)})$ is normal if and only if $(CL(X),\tau_{w(d)})$ is metrizable.

\end{document}